\documentclass[11pt,a4paper]{article}
\usepackage[T1]{fontenc}
\usepackage{lmodern}
\usepackage{amsmath,amssymb,amsthm,mathtools}
\usepackage{xcolor}
\usepackage{geometry}
\usepackage{microtype}
\usepackage{enumitem}
\usepackage{hyperref}
\usepackage{bm}
\usepackage{mathrsfs}
\hypersetup{colorlinks=true,linkcolor=blue,citecolor=blue,urlcolor=blue}
\allowdisplaybreaks

\newtheorem{theorem}{Theorem}[section]
\newtheorem{lemma}[theorem]{Lemma}
\newtheorem{proposition}[theorem]{Proposition}
\newtheorem{corollary}[theorem]{Corollary}
\newtheorem{principle}[theorem]{Principle}
\theoremstyle{remark}
\newtheorem{remark}[theorem]{Remark}

\newcommand{\norm}[1]{\left\lVert #1\right\rVert}
\newcommand{\abs}[1]{\left|#1\right|}
\newcommand{\one}{\mathbf 1}
\newcommand{\Ave}{\operatorname{Ave}}
\newcommand{\supp}{\operatorname{supp}}
\newcommand{\sgn}{\operatorname{sgn}}
\newcommand{\hooknot}{\mathrel{\not\hookrightarrow}}

\title{\textbf{Nonembedding of Lorentz sequence spaces into Lorentz function spaces}}
\author{Junxiang Qi\textsuperscript{1},\; Qi Liu\textsuperscript{1,*}\\[0.4em]
\small \textsuperscript{1}School of Mathematics and Statistics, Anqing Normal University, Anqing 246133, P. R. China\\
\small Emails: y25060009@stu.aqnu.edu.cn, liuq67@aqnu.edu.cn}
\date{}

\begin{document}
\maketitle

\begin{abstract}
{Let $0<p, q<\infty$ with $p \neq q$, and suppose that $p<2$ or $q \leq 1$. We prove that the Lorentz sequence space $\ell_ {p, q}$ does not isomorphically embed into $L_ {p, q}(0,1)$. Consequently, $L_ {p, q}(0, \infty)$ does not isomorphically embed into $L_ {p, q}(0,1)$.
}
\end{abstract}

\noindent\textbf{Keywords.} $L_{p,q}$-spaces; quasi-Banach spaces; Isomorphic embedding

\noindent\textbf{2020 Mathematics Subject Classification.} 46E30

\section{Introduction}
Lorentz spaces $L_{p,q}$, originating in the work of Lorentz \cite{Lorentz1950,Lorentz1951}, have become an important class of function spaces with substantial applications in harmonic analysis, interpolation theory, and related areas; see, for instance, \cite{KozonoYamazaki,NursultanovTikhonov,Billingsley,Dilworth} and the references therein. Let $(\Omega,\Sigma,\mu)$ be a measure space. For parameters $0<p<\infty$ and $0<q<\infty$, the Lorentz space $L_{p,q}(\Omega)$ consists of all measurable functions $f$ on $\Omega$ for which
\[
 \norm{f}_{p,q}=\left(\int_0^\infty [t^{1/p}f^*(t)]^q\,\frac{dt}{t}\right)^{1/q},
\]
where $f^*$ denotes the nonincreasing rearrangement of $|f|$. Given parameters $0<p<\infty$ and $0<q\le\infty$, an identical rearrangement-based construction yields two distinct families of spaces: the discrete sequence space $\ell_{p,q}$ associated with counting measure, alongside the continuous function spaces $L_{p,q}(0,1)$ and $L_{p,q}(0,\infty)$ equipped with Lebesgue measure. This formal similarity naturally leads to the structural question of determining which of these spaces can be isomorphically embedded into one another. Embedding problems also have applications in other branches of mathematics, but their meanings differ across different branches of mathematics; see the relevant literature \cite{SickelTriebel1995,SeegerTrebels2019,SchoelppleSteinwart2026}.

Prior literature has resolved this embedding question for the locally convex parameter region $1<p<\infty$, $1\le q<\infty$, $p\neq q$. Kuryakov and Sukochev demonstrated that $\ell_{p,q}$ cannot be isomorphically embedded into $L_{p,q}(0,1)$, leaving only the borderline case $p=2$, $1\leq q<2$ unaddressed at that time \cite{KuryakovSukochev}. Their analysis further established, under the same parameter restrictions, that $L_{p,q}(0,\infty)$ does not embed isomorphically into $L_{p,q}(0,1)\oplus\ell_{p,q}$ within the same parameter bounds. This lingering critical parameter regime $p=2$, $1\le q<2$ was fully resolved in \cite{SadovskayaSukochev}. Subsequently, Huang and Sukochev \cite{HuangSukochev} extended the isomorphic classification of Lorentz spaces $L_{p,q}$, previously established for resonant measure spaces, to general $\sigma$-finite measure spaces for $1<p<\infty$, $1\le q<\infty$, and $p\neq q$. More recently, Huang et al. \cite{HuangEtAl} generalized the corresponding nonembedding theorem to Lorentz spaces $L_{p,q}(\mathcal M,\tau)$ associated with noncommutative probability spaces.

All aforementioned classification results are confined to locally convex parameter regimes. The principal new contribution of the present work is to cross the local-convexity threshold and establish a unified nonembedding theorem on the parameter region
\[
\{(p,q):0<p,q<\infty,\ p\neq q,\ p<2\ 	or \ q\le1\}.
\]
Below the quadratic threshold, the decisive common-profile estimate depends only on $p<2$, whereas the splitting and disjoint-sequence ingredients remain available for all $0<p,q<\infty$. Thus the entire subcritical range $0<p<2$, $p\neq q$, is handled by one mechanism, including the genuinely nonlocally convex cases in which $p<1$ or $q<1$.

For $0<q\le1$ we then treat the complementary range $p\ge2$. The subcritical argument cannot simply be continued across $p=2$, since its strict $o(n^{1/p})$ common-profile estimate disappears at the quadratic threshold. When $p>2$, a different two-branch argument based on the same splitting principle, the finite-measure inclusion $L_{p,q}(0,1)\hookrightarrow L_2(0,1)$, and $q$-subadditivity yields the desired nonembedding. At $p=2$, the unweighted common-profile estimate is critical rather than strict; the missing separation is recovered by testing the image sequence with the weights $k^{-1/2}$. Their $\ell_{2,q}$ norm grows like a logarithmic power, while a family dominated by one rearrangement profile has strictly smaller weighted Rademacher growth. This complementary analysis is given in Section~6.

The quasi-Banach setting has received considerable attention in the study of function and sequence spaces. Sukochev and Xu investigated embedding problems for noncommutative $L_p$-spaces in the range $0<p<1$, showing that several non-embedding phenomena persist beyond the Banach setting. In the framework of rearrangement-invariant spaces, Kami\'nska and Raynaud studied copies of $\ell_p$ and $c_0$ in quasi-normed Orlicz--Lorentz sequence spaces and later developed a more general theory for symmetrized quasi-Banach lattices, including classical Lorentz spaces $L_{p,q}$ as important examples. These results provide a natural background for studying embedding and isomorphic classification problems for Lorentz spaces in the nonlocally convex range \cite{KaminskaRaynaud2007,SukochevXu,KaminskaRaynaud2009}. For further results on Lorentz spaces in the quasi-Banach setting, we refer the reader to the relevant literature \cite{NekvindaPesa,MusilovaEtAl,Grafakos}.

The main difference between the present proof and the earlier arguments is that the latter are essentially Banach-space methods. 
Kuryakov and Sukochev~\cite{KuryakovSukochev} established the corresponding nonembedding results in the Banach setting by combining a weakly-null subsequence splitting argument with unconditionality and standard Banach-space estimates. Their approach, however, relies essentially on local convexity and therefore does not extend directly to the genuinely quasi-Banach cases considered here.
The endpoint $p=2$, $1\le q<2$ already required a different method: Sadovskaya and Sukochev~\cite{SadovskayaSukochev} introduced the weighted Rademacher operator with coefficients $k^{-1/2}$ and used interpolation between $L_{2,1}$ and $L_2$ to obtain a logarithmic contradiction. This argument is tied to the Banach range $q\ge1$ and therefore does not cover $0<q<1$. Later works enlarged the classification to general $\sigma$-finite measure spaces~\cite{HuangSukochev} and to noncommutative probability spaces~\cite{HuangEtAl}, but their main nonembedding theorems still remain in the locally convex region $1<p<\infty$, $1\le q<\infty$. The present manuscript instead develops a mechanism that survives beyond local convexity: Aoki--Rolewicz renorming replaces the reverse triangle inequality, a quasi-Banach subsequence splitting principle separates the image sequence into a common-profile part, a disjoint part, and a norm-null error, and for $p<2$ the common-profile part satisfies the strict estimate $o(n^{1/p})$. This strict growth separation is the key new ingredient. For $p>2$, $0<q\le1$, the proof uses the finite-measure inclusion into $L_2$ together with $q$-subadditivity, while the critical case $p=2$ is recovered through the weights $k^{-1/2}$ by a direct quasi-Banach logarithmic estimate. Thus the novelty lies not merely in a larger parameter range, but in replacing Banach-space arguments that fail without local convexity by a proof scheme specifically adapted to quasi-Banach Lorentz spaces.

Our unified main result is as follows.
\begin{theorem}\label{thm:main}
Let $0<p,q<\infty$ with $p\neq q$. If
\[
0<p<2\qquad\text{or}\qquad 0<q\le1,
\]
then there is no linear isomorphic embedding
\[
T:\ell_{p,q}\longrightarrow L_{p,q}(0,1).
\]
\end{theorem}

The exclusion $p\neq q$ is sharp: when $q=p$, disjoint normalized indicators give an isometric copy of $\ell_p$ inside $L_p(0,1)$.

The subcritical part $p<2$ of Theorem~\ref{thm:main} is driven by a quantitative statement that may be useful beyond the present classification problem. Let $(r_k)$ be the Rademacher functions. If $(u_k)$ is a family in $L_{p,q}(0,1)$ whose rearrangements are dominated by one function $u$, set
\[
R_s(u_1,\ldots,u_n)=\left(\Ave_{\varepsilon_k=\pm1}\norm{\sum_{k=1}^n\varepsilon_k u_k}_{p,q}^s\right)^{1/s}.
\]
For $0<p<2$ and a suitable $s>0$ we prove
\begin{equation}\label{eq:intro-profile}
R_s(u_1,\ldots,u_n)=o(n^{1/p}).
\end{equation}
The bounded part of the common profile contributes only $O(n^{1/2})$. The tail is controlled at the natural $n^{1/p}$ scale by an anti-majorization estimate in $L_{p/2,q/2}$. Approximation of $u$ by bounded functions then turns the big-$O$ estimate into the little-$o$ estimate in \eqref{eq:intro-profile}. This is the technical point at which $p<2$ enters.

A subsequence splitting principle decomposes the image of the unit vector basis of $\ell_{p,q}$ into a common-profile part, a disjoint part and a norm-null perturbation. The first part satisfies \eqref{eq:intro-profile}. A seminormalized disjoint part has an $\ell_q$ subsequence. On the other hand, unconditionality of the source basis forces every randomized sum of the image sequence to have size comparable to
\[
\norm{\sum_{k=1}^n e_k}_{p,q}\asymp n^{1/p}.
\]
Consequently the disjoint part would have to grow simultaneously like $n^{1/q}$ and $n^{1/p}$, which is impossible when $p\neq q$.

\section{Preliminaries}
Let $(\Omega,\Sigma,\mu)$ be a $\sigma$-finite measure space. For a measurable function $f$ define
\[
d_f(\lambda)=\mu\{|f|>\lambda\},\qquad \lambda>0,
\]
and its decreasing rearrangement
\[
f^*(t)=\inf\{\lambda>0:d_f(\lambda)\le t\},\qquad t>0.
\]
For $0<p<\infty$ and $0<q<\infty$ we use the quasi-norm
\begin{equation}\label{eq:lorentz}
\norm{f}_{p,q}=\left(\int_0^{\infty}[t^{1/p}f^*(t)]^q\,\frac{dt}{t}\right)^{1/q},
\end{equation}
with the usual interpretation of the upper endpoint. For $q=\infty$,
\begin{equation}\label{eq:weak}
\norm{f}_{p,\infty}=\sup_{t>0}t^{1/p}f^*(t).
\end{equation}
Harmless normalization factors, such as $(q/p)^{1/q}$, are omitted. All embedding assertions are invariant under this choice.

For a scalar sequence $a=(a_k)$, let $(a_k^*)$ be the decreasing rearrangement of $(|a_k|)$. We take
\begin{equation}\label{eq:sequence}
\norm{a}_{\ell_{p,q}}=
\left(\sum_{k=1}^\infty (a_k^*)^q\bigl[k^{q/p}-(k-1)^{q/p}\bigr]\right)^{1/q}
\asymp
\left(\sum_{k=1}^\infty (a_k^*)^q k^{q/p-1}\right)^{1/q}.
\end{equation}
Thus the unit vector basis $(e_k)$ of $\ell_{p,q}$ is normalized, $1$-unconditional and symmetric. Its fundamental function is
\begin{equation}\label{eq:fundamental}
\varphi_{p,q}(n)=\norm{\sum_{k=1}^n e_k}_{\ell_{p,q}}\asymp n^{1/p},\qquad n\ge1.
\end{equation}
For $q=\infty$ the same formula holds with $\norm{a}_{p,\infty}=\sup_k k^{1/p}a_k^*$.

When $0<q<\infty$, $L_{p,q}$ has order-continuous quasi-norm. In particular, $L_\infty(0,1)$ is dense in $L_{p,q}(0,1)$. The space $L^0_{p,\infty}(0,1)$ is, by definition, the closure of $L_\infty(0,1)$ in $L_{p,\infty}(0,1)$, and $\ell^0_{p,\infty}$ is the closure of $c_{00}$ in $\ell_{p,\infty}$. These characterizations go back to the standard theory of Marcinkiewicz spaces; see \cite{BennettSharpley,KreinPetuninSemenov}.

Recall that two (quasi-) Banach spaces $X$ and $Y$ are said to be isomorphic if there exists an invertible bounded operator from $X$ onto $Y$. 
We say that $X$ embeds isomorphically into $Y$, and write
$
X\hookrightarrow Y,
$
if there exists a  linear subspace $Z\subset Y$ such that $Z$ is isomorphic to $X$. Equivalently, there exists a linear operator $T:X\to Y$
such that
$
C(T)^{-1}\Vert x\Vert _X\le \Vert Tx\Vert _Y\le C(T)\Vert x\Vert _X,
 x\in X, 
$
$C(T)$ is called the embedding constant of $X$ into $Y$.
If no such embedding exists,
we write 
$X\hooknot Y$.

We shall repeatedly use the Aoki--Rolewicz renorming theorem \cite{Aoki,Rolewicz}: every quasi-Banach space $X$ admits an equivalent quasi-norm, denoted here by $|||x|||_X$, and a number $0<\rho\le1$ such that
\begin{equation}\label{eq:aoki}
|||x+y|||_X^\rho\le |||x|||_X^\rho+|||y|||_X^\rho.
\end{equation}
Consequently, if $\norm{x_n}_X=o(A_n)$ and $\norm{x_n+y_n}_X\ge cA_n$, then
\begin{equation}\label{eq:reverse}
\norm{y_n}_X\ge c'A_n
\end{equation}
for all sufficiently large $n$. This elementary observation replaces every use of the reverse triangle inequality below.

We also use the small perturbation principle for basic sequences in a quasi-Banach space. If $(x_k)$ is a basic sequence with coordinate functionals $(x_k^*)$ and
\begin{equation}\label{eq:smallpert}
\sum_k\norm{x_k^*}\,\norm{x_k-y_k}<1,
\end{equation}
then $(y_k)$ is equivalent to $(x_k)$. The standard Neumann-series proof holds after an Aoki--Rolewicz renorming; see \cite[Chapter~1]{AlbiacKalton}. We shall use it only after passing to a subsequence for which the perturbations have summable $\rho$-powers.

Let $X$ be a quasi-Banach lattice. For a finite family $(f_k)_{k=1}^n$ put
\[
S_2(f_1,\ldots,f_n)=\left(\sum_{k=1}^n|f_k|^2\right)^{1/2}.
\]
The lattice-valued form of Khintchine's inequality says that, for every $0<s<\infty$,
\begin{equation}\label{eq:khintchine}
\left(\Ave_{\varepsilon_k=\pm1}\abs{\sum_{k=1}^n\varepsilon_k f_k}^{s}\right)^{1/s}\asymp_s S_2(f_1,\ldots,f_n)
\end{equation}
in the lattice order. This follows by applying the homogeneous lattice functional calculus to the scalar Khintchine inequality. A convenient modern reference, including the quasi-Banach formulation, is \cite[Theorem~2.2]{AlbiacAnsorenaBello}; see also \cite{Kalton}.

For $q<\infty$, the Lorentz lattice $L_{p,q}$ is $r$-concave for every $r>\max\{p,q\}$. Thus the $L_{p,q}$ quasi-norm may be moved outside an $L_r$-average in the direction needed below. This finite-concavity argument is unavailable at the weak endpoint $q=\infty$, a distinction used later.

\begin{lemma}\label{lem:rconcavity}
Let $0<p<\infty$ and $0<q<\infty$. If $r>\max\{p,q,2\}$, there is $C=C(p,q,r)$ such that every finite family in $L_{p,q}$ satisfies
\begin{equation}\label{eq:rconcavity}
\left(\Ave_{\varepsilon_k=\pm1}\norm{\sum_{k=1}^n\varepsilon_k f_k}_{p,q}^{r}\right)^{1/r}
\le C\norm{S_2(f_1,\ldots,f_n)}_{p,q}.
\end{equation}
\end{lemma}
\begin{proof}
This is the standard combination of the $r$-concavity of Lorentz lattices for $r>\max\{p,q\}$ (see \cite[Chapter~2]{BennettSharpley}) with the lattice-valued Khintchine inequality \eqref{eq:khintchine}. We therefore use \eqref{eq:rconcavity} as a cited structural estimate rather than reproduce its classical proof.
\end{proof}

\section{Rademacher estimates}
The decisive estimate in the range $p<2$ is most transparent after squaring. Then the first Lorentz index becomes $p/2<1$, and majorization reverses the usual norm comparison. We record the precise form needed later.

If $f_1,\ldots,f_n$ are functions, let $\bigoplus_{k=1}^n f_k$ denote their disjoint sum on a measure space of total measure $n$ times the original one. Recall that $a\prec b$ means $\int_0^t a^*\le\int_0^t b^*$ for all $t>0$ and equality holds for the full integrals whenever both are integrable.

\begin{lemma}\label{lem:antimaj}
Let $0<p<2$ and $0<q<\infty$. Suppose $f_1,\ldots,f_n\in L_{p,q}(0,1)$ and there is $f\in L_{p,q}(0,1)$ such that $f_k^*\le f^*$ for every $k$. Then
\begin{equation}\label{eq:antimaj}
\norm{\left(\sum_{k=1}^n|f_k|^2\right)^{1/2}}_{p,q}\le C_{p,q}n^{1/p}\norm{f}_{p,q}.
\end{equation}
\end{lemma}
\begin{proof}
We first assume that all functions are bounded and have support of finite measure. Put
\[
A=\sum_{k=1}^n|f_k|^2,\qquad B=\bigoplus_{k=1}^n|f_k|^2.
\]
The elementary disjointification inequality gives $B\prec A$; see, for example, \cite[Chapter~II]{BennettSharpley}. Since $p/2<1$, the Lorentz quasi-norm in $L_{p/2,q/2}$ is anti-monotone under majorization:
\begin{equation}\label{eq:lorentzshim}
\norm{A}_{p/2,q/2}\le C_{p,q}\norm{B}_{p/2,q/2}.
\end{equation}
This is precisely the anti-monotone half of the Lorentz--Shimogaki principle; we invoke it from \cite[Theorem~1.1(b)]{CadilhacSukochevZanin} rather than reproduce its proof.

The assumption $f_k^*\le f^*$ and lattice monotonicity give
\[
\norm{B}_{p/2,q/2}\le \norm{\bigoplus_{k=1}^n|f|^2}_{p/2,q/2}
=n^{2/p}\norm{|f|^2}_{p/2,q/2}=n^{2/p}\norm{f}_{p,q}^2.
\]
Taking square roots proves \eqref{eq:antimaj} for the bounded case. Truncation and order continuity prove it in general.
\end{proof}

The next proposition is the quantitative core of the main theorem.
\begin{proposition}\label{prop:profile}
Let $0<p<2$ and $0<q<\infty$. Suppose that $(u_k)$ is a sequence in $L_{p,q}(0,1)$ and that $u_k^*\le u^*$ for some $u\in L_{p,q}(0,1)$. Fix $r>\max\{p,q,2\}$. Then
\begin{equation}\label{eq:profile}
\left(\Ave_{\varepsilon_k=\pm1}\norm{\sum_{k=1}^n\varepsilon_k u_k}_{p,q}^{r}\right)^{1/r}=o(n^{1/p}).
\end{equation}
\end{proposition}
\begin{proof}
Fix $\eta>0$. By order continuity, choose a decomposition $u=a+b$ with $a\in L_\infty(0,1)$ and $\norm{b}_{p,q}<\eta$. For instance, take $a=\sgn(u)\min\{|u|,M\}$ and $b=u-a$ for sufficiently large $M$. The rearrangement domination $u_k^*\le u^*$ permits equimeasurable decompositions $u_k=a_k+b_k$ with
\begin{equation}\label{eq:decomp}
\norm{a_k}_\infty\le M,\qquad b_k^*\le b^*,\qquad k\ge1.
\end{equation}
One can obtain these by cutting $u_k$ at the same distribution level as $u$; an arbitrarily small discrepancy at a level set is absorbed into $b_k$.

By Lemma~\ref{lem:rconcavity} and the pointwise bound $S_2(a_1,\ldots,a_n)\le Mn^{1/2}\one$, we have
\begin{equation}\label{eq:boundedpart}
\left(\Ave_\varepsilon\norm{\sum_{k=1}^n\varepsilon_k a_k}_{p,q}^{r}\right)^{1/r}
\le C_{p,q,r}Mn^{1/2}.
\end{equation}
Applying Lemma~\ref{lem:rconcavity} and Lemma~\ref{lem:antimaj} to $(b_k)$ gives
\begin{equation}\label{eq:tailpart}
\left(\Ave_\varepsilon\norm{\sum_{k=1}^n\varepsilon_k b_k}_{p,q}^{r}\right)^{1/r}
\le C_{p,q,r}n^{1/p}\norm{b}_{p,q}\le C_{p,q,r}\eta n^{1/p}.
\end{equation}
The Aoki--Rolewicz quasi-triangle inequality on the finite sign probability space, followed by the scalar $L_r$ triangle inequality, combines \eqref{eq:boundedpart}--\eqref{eq:tailpart} to yield
\[
\limsup_{n\to\infty}n^{-1/p}
\left(\Ave_\varepsilon\norm{\sum_{k=1}^n\varepsilon_k u_k}_{p,q}^{r}\right)^{1/r}
\le C_{p,q,r}\eta,
\]
because $p<2$. Letting $\eta\downarrow0$ proves \eqref{eq:profile}.
\end{proof}

\begin{remark}
The estimate $O(n^{1/p})$ alone contains no obstruction: it agrees with the fundamental function of the atomic source. The bounded-tail decomposition is what converts the square-function estimate into a strict asymptotic gap. This strictness is lost at $p=2$, where the bounded contribution in \eqref{eq:boundedpart} already has the critical size $n^{1/2}$.
\end{remark}

\section{Subsequence splitting}
We need a commutative subsequence splitting principle in a form that remains valid for nonlocally convex Lorentz spaces. The Banach version originates in the Kadec--Pe\l czy\'nski method and was developed for symmetric spaces in \cite{Weis,DoddsSemenovSukochev}. The formulation below is tailored to the order-continuous Lorentz scale.

\begin{lemma}\label{lem:envelope}
Let $0<p,q<\infty$, and let $(g_n)$ be a bounded sequence in $L_{p,q}(0,1)$ which is uniformly absolutely continuous in the following sense:
\begin{equation}\label{eq:uac}
\omega(\delta):=\sup_n\sup_{m(A)\le\delta}\norm{g_n\one_A}_{p,q}\longrightarrow0\qquad(\delta\downarrow0).
\end{equation}
Then $(g_n)$ has a subsequence, still denoted by $(g_n)$, and there exists $u\in L_{p,q}(0,1)$ such that
\[
g_n^*\le u^*\qquad(n\ge1).
\]
\end{lemma}
\begin{proof}
Put $h_n=g_n^*$. Boundedness gives $h_n(t)\lesssim_{p,q}t^{-1/p}$, and Helly selection yields a subsequence converging almost everywhere to a decreasing $h\in L_{p,q}(0,1)$. Condition \eqref{eq:uac}, applied to sets carrying the largest $\delta$-portion of $|g_n|$, gives uniform smallness of $h_n$ on $(0,\delta)$; Fatou gives the same for $h$. Splitting $(0,1)$ into $(0,\delta)$ and $[\delta,1]$, order continuity and bounded convergence then give $h_n\to h$ in $L_{p,q}$. Passing to a further subsequence with summable Aoki--Rolewicz powers and setting $U=h+\sum_n|h_n-h|$, we obtain $U\in L_{p,q}$ and $h_n\le U$ for every $n$. Hence $g_n^*=h_n\le U^*$, as required. This is the usual rearrangement-compactness argument in an order-continuous Lorentz lattice; compare \cite[Chapter~II]{BennettSharpley}.
\end{proof}
\addtocounter{equation}{5}

\begin{remark}
In the commutative finite-measure setting, Randrianantoanina's $E$-equi-integrability \cite[Definition~2.5]{Randrianantoanina2002} implies \eqref{eq:uac}. This is the standard decreasing-set consequence of the definition: otherwise a sequence of exceptional sets may be replaced by their decreasing tail unions, contradicting equi-integrability.
\end{remark}

We shall use the following convenient form of the Kadec--Pe\l czy\'nski type subsequence splitting principle \cite{KadecPelczynski}. It follows from \cite[Theorem~3.9]{Randrianantoanina2002}, together with the preceding remark and Lemma~\ref{lem:envelope}.

\begin{principle}\label{pr:splitting}
Let $0<p,q<\infty$, and let $(f_n)$ be bounded in $L_{p,q}(0,1)$. There are a subsequence, still denoted by $(f_n)$, and sequences $(u_n)$, $(v_n)$ and $(w_n)$ in $L_{p,q}(0,1)$ such that
\begin{enumerate}[label=(\roman*)]
\item $f_n=u_n+v_n+w_n$;
\item $u_n^*\le u^*$ for every $n$, for some fixed $u\in L_{p,q}(0,1)$;
\item the functions $v_n$ are pairwise disjointly supported;
\item $\norm{w_n}_{p,q}\to0$.
\end{enumerate}
Moreover, if $(f_n)$ is a basic sequence, the subsequence can be chosen so that $(u_n+v_n)$ is equivalent to $(f_n)$.
\end{principle}
\begin{proof}
Apply Randrianantoanina's subsequence splitting theorem \cite[Theorem~3.9]{Randrianantoanina2002}: after passage, $f_n=g_n+v_n$ with $(g_n)$ $L_{p,q}$-equi-integrable and $(v_n)$ disjoint. The preceding remark and Lemma~\ref{lem:envelope} give a fixed $u$ with $g_n^*\le u^*$. Thus one may take $u_n=g_n$ and $w_n=0$; the assertion for basic sequences is immediate.
\end{proof}

For later use, we retain the perturbative formulation above, although the present argument actually yields the stronger conclusion $w_n=0$.

We restate a classical result in what follows (\cite[Lemma 2.1]{CarothersDilworth}, additional references can be found in \cite{1,Dilworth,2,3,4,HuangEtAl}).

\begin{proposition}\label{prop:disjoint}
Suppose that $0<p,q<\infty$. Let $(f_n)$ be a sequence of unit vectors in $L_{p,q}$ such that $f_n^*\to0$ pointwise. Then some subsequence of $(f_n)$ is equivalent to the unit vector basis of $\ell_q$.
\end{proposition}

The following lemma is a direct consequence of Proposition~\ref{prop:disjoint}. We state it separately for later use and omit the proof.
\begin{lemma}\label{lem:disjointdichotomy}
Let $0<p,q<\infty$, and let $(v_n)$ be a bounded pairwise disjoint sequence in $L_{p,q}(0,1)$. After passing to a subsequence, exactly one of the following alternatives occurs:
\begin{enumerate}[label=(\alph*)]
\item $\norm{v_n}_{p,q}\to0$;
\item $(v_n)$ is seminormalized and has a subsequence equivalent to the unit vector basis of $\ell_q$.
\end{enumerate}
\end{lemma}

\begin{remark}
The conclusion is about a subsequence, not every disjoint sequence. This distinction is important: the measures and heights of a general disjoint family can encode additional weights. The gliding-hump passage removes that irrelevant local information and retains the outer Lorentz exponent $q$.
\end{remark}

\section{A unified nonembedding theorem below $p=2$}
We first prove the subcritical part of Theorem~\ref{thm:main}, namely the range $0<p<2$, $0<q<\infty$, $p\neq q$. The argument is short once the two structural ingredients have been separated, and it makes explicit why $p=2$ is the threshold for this particular mechanism.

\begin{proof}[Proof of the subcritical part of Theorem~\ref{thm:main}]
Assume $0<p<2$, $0<q<\infty$, and $p\neq q$. Assume, towards a contradiction, that $T:\ell_{p,q}\to L_{p,q}(0,1)$ is an isomorphic embedding. Set
\[
f_k=Te_k,\qquad k\ge1.
\]
The sequence $(f_k)$ is seminormalized and equivalent to the symmetric unit vector basis of $\ell_{p,q}$. Hence there is $K\ge1$ such that for every finite scalar family $(a_k)$ and every choice of signs,
\begin{equation}\label{eq:equiv}
K^{-1}\norm{(a_k)}_{\ell_{p,q}}
\le\norm{\sum_k\varepsilon_k a_kf_k}_{p,q}
\le K\norm{(a_k)}_{\ell_{p,q}}.
\end{equation}
Apply Principle~\ref{pr:splitting} and pass to a subsequence. We may write
\[
f_k=u_k+v_k+w_k,
\]
where $u_k^*\le u^*$ for one $u\in L_{p,q}(0,1)$, the $v_k$ are pairwise disjoint, and $w_k\to0$ in norm. By a further passage and the small perturbation principle, $(y_k)$ with $y_k=u_k+v_k$ is equivalent to $(f_k)$ and still satisfies \eqref{eq:equiv}, with a possibly different constant.

Since $p<2$, Proposition~\ref{prop:profile} gives, for any fixed $r>\max\{p,q,2\}$,
\begin{equation}\label{eq:Un}
U_n:=\left(\Ave_\varepsilon\norm{\sum_{k=1}^n\varepsilon_ku_k}_{p,q}^r\right)^{1/r}=o(n^{1/p}).
\end{equation}
On the other hand, \eqref{eq:equiv} and \eqref{eq:fundamental} imply
\begin{equation}\label{eq:Yn}
Y_n:=\left(\Ave_\varepsilon\norm{\sum_{k=1}^n\varepsilon_ky_k}_{p,q}^r\right)^{1/r}\asymp n^{1/p}.
\end{equation}
Use an Aoki--Rolewicz power in the product of the sign probability space and $L_{p,q}$. The reverse estimate \eqref{eq:reverse}, applied to $\sum\varepsilon_ky_k=\sum\varepsilon_ku_k+\sum\varepsilon_kv_k$, combines \eqref{eq:Un}--\eqref{eq:Yn} to give
\begin{equation}\label{eq:Vn}
V_n:=\left(\Ave_\varepsilon\norm{\sum_{k=1}^n\varepsilon_kv_k}_{p,q}^r\right)^{1/r}\asymp n^{1/p}.
\end{equation}
The norm-null alternative in Lemma~\ref{lem:disjointdichotomy} is incompatible with \eqref{eq:Vn} after another small perturbation passage. We may therefore assume that $(v_k)$ is equivalent to the unit vector basis of $\ell_q$. Since that basis is unconditional,
\begin{equation}\label{eq:ellqgrowth}
V_n\asymp\norm{\sum_{k=1}^n e_k}_{\ell_q}=n^{1/q}.
\end{equation}
Equations \eqref{eq:Vn} and \eqref{eq:ellqgrowth} force $n^{1/p}\asymp n^{1/q}$, and hence $p=q$, contrary to the hypothesis.
\end{proof}

\paragraph{Relation to the known range.}
This subcritical part of the main theorem is valid on the full range $0<p<2$ because its decisive common-profile step depends on $p<2$, not on local convexity. In the locally convex sector $1<p<2$, $1\le q<\infty$, $p\neq q$, it recovers previously known nonembedding results. Its genuinely new content lies in the nonlocally convex sector in which $p<1$ or $q<1$. The remaining portion of Theorem~\ref{thm:main}, namely $p\ge2$ with $0<q\le1$, is proved in Section~6.

Two consequences locate the result in the resonant classification picture.

\begin{corollary}
Let $0<p<2$, $0<q<\infty$, and $p\neq q$. Then
\[
L_{p,q}(0,\infty)\hooknot L_{p,q}(0,1)
\]
isomorphically. In particular the two spaces are not isomorphic.
\end{corollary}
\begin{proof}
The map
\[
(a_k)\longmapsto\sum_{k=1}^\infty a_k\one_{(k-1,k]}
\]
is an isomorphic, in fact normalization-dependent isometric, embedding of $\ell_{p,q}$ into $L_{p,q}(0,\infty)$. An embedding of the latter into $L_{p,q}(0,1)$ would contradict Theorem~\ref{thm:main}.
\end{proof}

\begin{corollary}
Let $0<p<2$, $0<q<\infty$, and $p\neq q$. If $S:L_{p,q}(0,\infty)\to L_{p,q}(0,1)$ is bounded, then the restriction of $S$ to the canonical unit-interval copy of $\ell_{p,q}$ is not an isomorphism on any subsequence of its symmetric basis whose closed span is isomorphic to $\ell_{p,q}$.
\end{corollary}
\begin{proof}
Otherwise that restriction would furnish the embedding forbidden by Theorem~\ref{thm:main}.
\end{proof}

\begin{remark}
The preceding corollary is deliberately stated only on the canonical atomic core. It does not claim a complete classification of all strictly singular operators on quasi-Banach Lorentz spaces. Such a claim would require additional ideal structure not used here.
\end{remark}

\begin{remark}
The restriction $p<2$ is a limitation of the argument in Sections~3--5, not an assertion that an embedding exists for $p\ge2$. In Proposition~\ref{prop:profile} the bounded part of the common profile has size $O(n^{1/2})$. This is $o(n^{1/p})$ precisely when $p<2$; at $p=2$ it is already of the critical size $n^{1/p}$, and for $p>2$ it dominates that scale. Thus the strict estimate $U_n=o(n^{1/p})$ that drives the contradiction is unavailable for $p\ge2$, and that range requires genuinely different arguments. Section~6 supplies one for $p>2$, $0<q\le1$, and a logarithmically weighted variant at the endpoint $p=2$.
\end{remark}

\section{The range $p\ge2$ when $0<q\le1$}
The proof below the quadratic threshold uses the strict asymptotic separation
\[
n^{1/2}=o(n^{1/p})\qquad(p<2).
\]
At $p=2$ this separation disappears, and for $p>2$ its direction is reversed. We therefore use a different argument in the latter range. The splitting principle and the disjoint alternative from Section~4 remain available without change.

\begin{lemma}\label{lem:qsubadd}
Let $0<q\le1$ and $p>q$. For every finite family $(h_j)$ in $L_{p,q}(0,1)$,
\begin{equation}\label{eq:qsubadd}
\norm{\sum_j h_j}_{p,q}^{q}\le C_{p,q}\sum_j\norm{h_j}_{p,q}^{q}.
\end{equation}
Consequently, for $x,z\in L_{p,q}(0,1)$,
\begin{equation}\label{eq:qreverse}
\norm{x}_{p,q}^{q}\le C_{p,q}\bigl(\norm{x+z}_{p,q}^{q}+\norm{z}_{p,q}^{q}\bigr).
\end{equation}
\end{lemma}
\begin{proof}
Since $0<q\le1$,
\[
\abs{\sum_jh_j}^{q}\le\sum_j|h_j|^{q}.
\]
Moreover, with the normalization used in \eqref{eq:lorentz},
\[
\norm{h}_{p,q}^{q}=\norm{|h|^q}_{p/q,1}.
\]
Because $p/q>1$, the Lorentz space $L_{p/q,1}$ is normable; see, for example, \cite[Chapter~2]{BennettSharpley}. Lattice monotonicity and the triangle inequality in an equivalent norm on $L_{p/q,1}$ therefore yield \eqref{eq:qsubadd}. Applying \eqref{eq:qsubadd} to $x=(x+z)-z$ gives \eqref{eq:qreverse}.
\end{proof}

\begin{lemma}\label{lem:commonprofile6}
Suppose that $(u_k)\subset L_{p,q}(0,1)$ satisfies
\[
u_k^*\le u^*\qquad(k\ge1)
\]
for some $u\in L_{p,q}(0,1)$, where $0<q<\infty$.
\begin{enumerate}[label=(\roman*)]
\item The family $(u_k)$ is uniformly absolutely continuous: if $m(A)\le\delta$, then
\begin{equation}\label{eq:common-uac}
\sup_k\norm{u_k\one_A}_{p,q}\le\norm{u^*\one_{(0,\delta)}}_{p,q}\longrightarrow0\qquad(\delta\downarrow0).
\end{equation}
\item If $p>2$, $0<q\le1$, and $\inf_k\norm{u_k}_{p,q}>0$, then
\begin{equation}\label{eq:L2lower}
\inf_k\norm{u_k}_2>0.
\end{equation}
\end{enumerate}
\end{lemma}
\begin{proof}
For (i), the support of $u_k\one_A$ has measure at most $\delta$, and rearrangement monotonicity gives
\[
(|u_k|\one_A)^*(t)\le u_k^*(t)\one_{(0,\delta)}(t)\le u^*(t)\one_{(0,\delta)}(t).
\]
Thus \eqref{eq:common-uac} follows from the order continuity of $L_{p,q}(0,1)$.

For (ii), the standard Lorentz inclusions on a finite measure space give
\[
L_{p,q}(0,1)\hookrightarrow L_p(0,1)\hookrightarrow L_2(0,1),
\]
since $q\le p$ and $p>2$; see \cite[Chapter~2]{BennettSharpley}. Suppose that a subsequence satisfies $\norm{u_k}_2\to0$. Then, for each $t\in(0,1)$,
\[
u_k^*(t)\le t^{-1/2}\norm{u_k}_2\longrightarrow0.
\]
Together with $u_k^*\le u^*$, dominated convergence in the defining Lorentz integral gives $\norm{u_k}_{p,q}\to0$, a contradiction.
\end{proof}

\begin{lemma}\label{lem:hilbertgrowth}
Let $(x_k)$ be a bounded sequence in a Hilbert space $H$ such that
\[
\inf_k\norm{x_k}_H>0.
\]
Then some subsequence, still denoted by $(x_k)$, satisfies
\begin{equation}\label{eq:hilbertgrowth}
\norm{\sum_{k=1}^n x_k}_H\ge cn^{1/2},\qquad n\ge1,
\end{equation}
for some $c>0$.
\end{lemma}
\begin{proof}
This is the standard Hilbert-space subsequence argument. After passing to a weakly convergent subsequence $x_k\rightharpoonup x$, a nonzero weak limit gives even linear growth in the direction $x$. If $x=0$, a gliding-hump selection makes the off-diagonal inner products summable, and the polarization identity then gives $\norm{\sum_{k=1}^n x_k}_H^2\gtrsim n$. We record only the consequence \eqref{eq:hilbertgrowth} needed below.
\end{proof}

\begin{theorem}\label{thm:above2}
Let
\[
p>2,\qquad0<q\le1.
\]
Then
\[
\ell_{p,q}\hooknot L_{p,q}(0,1)
\]
isomorphically.
\end{theorem}
\begin{proof}
Assume that $T:\ell_{p,q}\to L_{p,q}(0,1)$ is an isomorphic embedding and put $f_k=Te_k$. Apply Principle~\ref{pr:splitting}. After passing to a subsequence and removing the norm-null perturbation by the small perturbation principle, we may write
\begin{equation}\label{eq:yk6}
y_k=u_k+v_k,
\end{equation}
where $(y_k)$ is equivalent to the canonical basis of $\ell_{p,q}$, the functions $v_k$ are pairwise disjointly supported, and
\[
u_k^*\le u^*\qquad(k\ge1)
\]
for one fixed $u\in L_{p,q}(0,1)$. In particular,
\begin{equation}\label{eq:sumyk6}
\norm{\sum_{k=1}^n y_k}_{p,q}\asymp n^{1/p}.
\end{equation}

Apply Lemma~\ref{lem:disjointdichotomy} to $(v_k)$.

\emph{Case 1: $\norm{v_k}_{p,q}\to0$.} Pass to a further subsequence so rapidly that the small perturbation principle applies to $u_k=y_k-v_k$. Then $(u_k)$ is equivalent to $(y_k)$ and hence to the canonical basis of $\ell_{p,q}$. Thus $(u_k)$ is seminormalized in $L_{p,q}$, and Lemma~\ref{lem:commonprofile6}(ii) yields
\[
\inf_k\norm{u_k}_2>0.
\]
The inclusion $L_{p,q}(0,1)\hookrightarrow L_2(0,1)$ also makes $(u_k)$ bounded in $L_2$. By Lemma~\ref{lem:hilbertgrowth}, after another passage,
\[
\norm{\sum_{k=1}^n u_k}_2\ge cn^{1/2}.
\]
On the other hand, equivalence with the symmetric basis of $\ell_{p,q}$ and the continuous inclusion into $L_2$ give
\[
cn^{1/2}\le\norm{\sum_{k=1}^n u_k}_2
\le C\norm{\sum_{k=1}^n u_k}_{p,q}
\le C'n^{1/p},
\]
which is impossible because $p>2$.

\emph{Case 2: $(v_k)$ is seminormalized and equivalent, after passage, to the canonical basis of $\ell_q$.} Let $A_k=\supp v_k$ and
\[
A^{(N)}=\bigcup_{k\ge N}A_k.
\]
Since the $A_k$ are pairwise disjoint subsets of $(0,1)$, $m(A^{(N)})\to0$. By Lemma~\ref{lem:commonprofile6}(i), given $\eta>0$ we may choose $N$ so large that
\begin{equation}\label{eq:tailA}
\sup_j\norm{u_j\one_{A^{(N)}}}_{p,q}<\eta.
\end{equation}
For finitely supported scalars $(a_k)$ with support in $\{k\ge N\}$, set
\[
V=\sum_k a_kv_k,\qquad U_A=\sum_k a_ku_k\one_{A^{(N)}}.
\]
The equivalence of $(v_k)$ with the $\ell_q$ basis gives
\begin{equation}\label{eq:Vellq}
\norm{V}_{p,q}\ge c_0\norm{(a_k)}_{\ell_q}.
\end{equation}
By Lemma~\ref{lem:qsubadd} and \eqref{eq:tailA},
\begin{equation}\label{eq:UA}
\norm{U_A}_{p,q}^{q}\le C_{p,q}\eta^q\sum_k|a_k|^q.
\end{equation}
Moreover,
\[
\left(\sum_k a_ky_k\right)\one_{A^{(N)}}=V+U_A.
\]
Using \eqref{eq:qreverse}, \eqref{eq:Vellq}, and \eqref{eq:UA}, and then taking $\eta$ sufficiently small, we obtain
\[
\norm{V+U_A}_{p,q}\ge c_1\norm{(a_k)}_{\ell_q}.
\]
Lattice monotonicity therefore yields
\begin{equation}\label{eq:taildom}
\norm{\sum_k a_ky_k}_{p,q}\ge c_1\norm{(a_k)}_{\ell_q}.
\end{equation}
Since $(y_k)$ is equivalent to the $\ell_{p,q}$ basis, \eqref{eq:taildom} implies
\[
\norm{(a_k)}_{\ell_{p,q}}\gtrsim\norm{(a_k)}_{\ell_q}
\]
on a tail. Taking $a_k=1$ on $n$ consecutive indices gives
\[
n^{1/p}\gtrsim n^{1/q},
\]
which is impossible because $q\le1<p$.

Both alternatives are impossible, so the embedding does not exist.
\end{proof}

\begin{corollary}\label{cor:above2infty}
Let $p>2$ and $0<q\le1$. Then
\[
L_{p,q}(0,\infty)\hooknot L_{p,q}(0,1)
\]
isomorphically. In particular, the two spaces are not isomorphic.
\end{corollary}
\begin{proof}
The canonical disjoint-indicator map embeds $\ell_{p,q}$ isomorphically into $L_{p,q}(0,\infty)$. An embedding of $L_{p,q}(0,\infty)$ into $L_{p,q}(0,1)$ would contradict Theorem~\ref{thm:above2}.
\end{proof}

\subsection{The case $p=2$}
The argument above deliberately separates $p>2$ from $p=2$. At $p=2$ the Hilbert-space growth in Case~1 is of the same order as the fundamental function of $\ell_{2,q}$, so an unweighted partial-sum argument cannot exclude the common-profile branch. The endpoint is nevertheless recoverable by using the critical coefficient sequence $(k^{-1/2})$. Put
\[
H_n=\sum_{k=1}^n\frac1k\asymp\log(en).
\]
By \eqref{eq:sequence}, for every choice of signs,
\[
\norm{\sum_{k=1}^n\frac{\varepsilon_k}{\sqrt{k}}e_k}_{\ell_{2,q}}\asymp_q H_n^{1/q}.
\]
The essential point is that a common-profile family has strictly smaller weighted Rademacher growth.

We first record the elementary shell estimate that creates the logarithmic gain.

\begin{lemma}\label{lem:logshell}
Let $0<s<1$. Suppose that $y\ge0$ is measurable and $a\le y\le b$ on $\supp y$, where $0<a\le b<\infty$. Then
\begin{equation}\label{eq:logshell}
\norm{y}_{1,s}^{s}\le C_s\bigl(1+\log(b/a)\bigr)^{1-s}\norm{y}_1^{s}.
\end{equation}
\end{lemma}
\begin{proof}
Let
\[
A_j=\{2^ja<y\le2^{j+1}a\}
\]
for the nonempty shells between $a$ and $b$, and put $\alpha_j=2^ja\,m(A_j)$. There are at most $N\le C(1+\log(b/a))$ such shells. Since $y\one_{A_j}\le2^{j+1}a\one_{A_j}$, the definition of $L_{1,s}$ and the $s$-subadditivity of its quasi-norm give
\[
\norm{y}_{1,s}^{s}\le C_s\sum_j\alpha_j^s
\le C_sN^{1-s}\left(\sum_j\alpha_j\right)^s.
\]
Moreover, $\sum_j\alpha_j\le\norm{y}_1$, proving \eqref{eq:logshell}.
\end{proof}

\begin{lemma}\label{lem:boundedsupport}
Let $0<q\le1$, $|h_k|\le M$, and $m(\supp h_k)\le t$ for $1\le k\le n$. Then
\begin{equation}\label{eq:boundedsupport}
\norm{\left(\sum_{k=1}^n\frac{|h_k|^2}{k}\right)^{1/2}}_{2,q}^{q}
\le C_qM^qt^{q/2}H_n.
\end{equation}
\end{lemma}
\begin{proof}
By homogeneity take $M=1$, set $s=q/2$, and
\[
X=\sum_{k=1}^n\frac{|h_k|^2}{k}.
\]
Then $0\le X\le H_n\le n$, $\norm{X}_1\le tH_n$, and $m(\supp X)\le\min\{1,nt\}$. Split
\[
y=X\one_{\{X>1/n\}},\qquad z=X\one_{\{X\le1/n\}}.
\]
The values of $y$ lie between $1/n$ and $n$. Lemma~\ref{lem:logshell} and $H_n\asymp\log(en)$ imply
\[
\norm{y}_{1,s}^{s}\le C_s(\log(en))^{1-s}(tH_n)^s\le C_st^sH_n.
\]
For the low part, direct comparison with a constant on its support yields
\[
\norm{z}_{1,s}^{s}\le C_sn^{-s}\min\{1,nt\}^s\le C_st^s.
\]
Finally, using $\norm{X^{1/2}}_{2,q}^{q}=\norm{X}_{1,s}^{s}$ gives \eqref{eq:boundedsupport}.
\end{proof}

\begin{proposition}\label{prop:criticalprofile}
Let $0<q\le1$ and let $f_1,\ldots,f_n\in L_{2,q}(0,1)$. If $f_k^*\le f^*$ for one $f\in L_{2,q}(0,1)$ and every $k$, then
\begin{equation}\label{eq:criticalprofile}
\norm{\left(\sum_{k=1}^n\frac{|f_k|^2}{k}\right)^{1/2}}_{2,q}^{q}
\le C_qH_n\norm{f}_{2,q}^{q}.
\end{equation}
\end{proposition}
\begin{proof}
For $j\in\mathbb Z$ put
\[
f_{k,j}=f_k\one_{\{2^j<|f_k|\le2^{j+1}\}}.
\]
Then $|f_{k,j}|\le2^{j+1}$ and
\[
m(\supp f_{k,j})\le d_{f_k}(2^j)\le d_f(2^j).
\]
Pointwise Minkowski for the Euclidean norm, followed by $q$-subadditivity in $L_{2,q}$ and Lemma~\ref{lem:boundedsupport}, gives
\[
\norm{\left(\sum_{k=1}^n\frac{|f_k|^2}{k}\right)^{1/2}}_{2,q}^{q}
\le C_q\sum_{j\in\mathbb Z}
\norm{\left(\sum_{k=1}^n\frac{|f_{k,j}|^2}{k}\right)^{1/2}}_{2,q}^{q}
\le C_qH_n\sum_{j\in\mathbb Z}2^{jq}d_f(2^j)^{q/2}.
\]
The standard dyadic discretization of the Lorentz quasi-norm gives
\[
\sum_{j\in\mathbb Z}2^{jq}d_f(2^j)^{q/2}\asymp_q\norm{f}_{2,q}^{q},
\]
which proves \eqref{eq:criticalprofile}.
\end{proof}

\begin{theorem}\label{thm:strictendpoint}
Let $0<q\le1$ and suppose that $f_k^*\le f^*$ for all $k$, where $f\in L_{2,q}(0,1)$. Then
\begin{equation}\label{eq:strictsquare}
\norm{\left(\sum_{k=1}^n\frac{|f_k|^2}{k}\right)^{1/2}}_{2,q}
=o\bigl(H_n^{1/q}\bigr).
\end{equation}
Consequently, for every fixed $r>\max\{2,q\}$,
\begin{equation}\label{eq:strictrademacher}
\left(\Ave_{\varepsilon_k=\pm1}
\norm{\sum_{k=1}^n\frac{\varepsilon_k f_k}{\sqrt{k}}}_{2,q}^{r}\right)^{1/r}
=o\bigl(H_n^{1/q}\bigr).
\end{equation}
\end{theorem}
\begin{proof}
Fix $M>0$ and decompose
\[
a_k=f_k\one_{\{|f_k|\le M\}},\qquad b_k=f_k\one_{\{|f_k|>M\}}.
\]
The bounded part satisfies pointwise
\[
\left(\sum_{k=1}^n\frac{|a_k|^2}{k}\right)^{1/2}\le MH_n^{1/2}\one.
\]
Hence, since $q\le1$,
\begin{equation}\label{eq:boundedendpoint}
H_n^{-1/q}\norm{\left(\sum_{k=1}^n\frac{|a_k|^2}{k}\right)^{1/2}}_{2,q}
\le C_qM H_n^{1/2-1/q}\longrightarrow0.
\end{equation}
Let
\[
f^{(M)}=f^*\one_{\{f^*>M\}}.
\]
Distribution domination gives $b_k^*\le(f^{(M)})^*$: below level $M$ both distribution functions are controlled by $d_f(M)$, and above level $M$ the assertion is inherited from $f_k^*\le f^*$. Proposition~\ref{prop:criticalprofile} therefore gives
\begin{equation}\label{eq:tailendpoint}
H_n^{-1/q}\norm{\left(\sum_{k=1}^n\frac{|b_k|^2}{k}\right)^{1/2}}_{2,q}
\le C_q\norm{f^{(M)}}_{2,q}.
\end{equation}
Because $q<\infty$, $L_{2,q}$ is order continuous and the right-hand side tends to zero as $M\to\infty$. Since the full square function is bounded by the sum of the two square functions, \eqref{eq:boundedendpoint}--\eqref{eq:tailendpoint} prove \eqref{eq:strictsquare}. Formula \eqref{eq:strictrademacher} follows from Lemma~\ref{lem:rconcavity}, applied with $p=2$ and $g_k=f_k/\sqrt{k}$.
\end{proof}

\begin{remark}
Proposition~\ref{prop:criticalprofile} alone has the same $H_n^{1/q}$ scale as the source norm. The truncation argument in Theorem~\ref{thm:strictendpoint} creates strictness: bounded pieces grow only like $H_n^{1/2}$, whereas tails are uniformly small in $L_{2,q}$. This is where the restriction $q<2$ enters, and it is more than sufficient for the present range $0<q\le1$.
\end{remark}

\begin{theorem}\label{thm:endpoint}
Let $0<q\le1$. There is no linear isomorphic embedding
\[
\ell_{2,q}\longrightarrow L_{2,q}(0,1).
\]
\end{theorem}
\begin{proof}
Assume that $T:\ell_{2,q}\to L_{2,q}(0,1)$ is an isomorphic embedding and put $x_k=Te_k$. Apply Principle~\ref{pr:splitting}; after a subsequence and a small perturbation, we may work with
\[
y_k=u_k+v_k\sim e_k,\qquad u_k^*\le u^*,
\]
where the $v_k$ are disjoint. Fix $r>\max\{2,q\}$ and set
\[
Y_n=\left(\Ave_\varepsilon\norm{\sum_{k=1}^n\frac{\varepsilon_ky_k}{\sqrt{k}}}_{2,q}^{r}\right)^{1/r},
\quad
U_n=\left(\Ave_\varepsilon\norm{\sum_{k=1}^n\frac{\varepsilon_ku_k}{\sqrt{k}}}_{2,q}^{r}\right)^{1/r},
\]
\[
V_n=\left(\Ave_\varepsilon\norm{\sum_{k=1}^n\frac{\varepsilon_kv_k}{\sqrt{k}}}_{2,q}^{r}\right)^{1/r}.
\]
The equivalence $(y_k)\sim(e_k)$ and the critical coefficient estimate give
\begin{equation}\label{eq:endpointYn}
Y_n\asymp_q H_n^{1/q}.
\end{equation}
Theorem~\ref{thm:strictendpoint} gives
\begin{equation}\label{eq:endpointUn}
U_n=o(H_n^{1/q}).
\end{equation}

If $\norm{v_k}_{2,q}\to0$, pass to a further subsequence so rapidly that $(v_k)$ is a small perturbation. Then $(u_k)\sim(y_k)\sim(e_k)$, and hence every choice of signs satisfies
\[
\norm{\sum_{k=1}^n\frac{\varepsilon_ku_k}{\sqrt{k}}}_{2,q}\ge cH_n^{1/q},
\]
contradicting \eqref{eq:endpointUn}.

It remains to consider the seminormalized disjoint branch. By Lemma~\ref{lem:disjointdichotomy}, after passage $(v_k)$ is equivalent to the unit vector basis of $\ell_q$. Therefore signs do not affect its asymptotic size and
\begin{equation}\label{eq:endpointVlower}
V_n\asymp_q\left(\sum_{k=1}^n k^{-q/2}\right)^{1/q}
\asymp_q n^{1/q-1/2}.
\end{equation}
On the other hand, $v_k=y_k-u_k$. The power triangle inequality in $L_{2,q}$, followed by the ordinary triangle inequality in the finite $L_r$ average, gives from \eqref{eq:endpointYn} and \eqref{eq:endpointUn}
\begin{equation}\label{eq:endpointVupper}
V_n\le C_q(Y_n+U_n)=O(H_n^{1/q}).
\end{equation}
Since $q<2$,
\[
\frac{n^{1/q-1/2}}{H_n^{1/q}}\longrightarrow\infty,
\]
so \eqref{eq:endpointVlower} and \eqref{eq:endpointVupper} are incompatible. Both structural alternatives lead to a contradiction, and the proof is complete.
\end{proof}

\begin{corollary}\label{cor:endpointinfty}
For $0<q\le1$, there is no isomorphic embedding
\[
L_{2,q}(0,\infty)\longrightarrow L_{2,q}(0,1).
\]
\end{corollary}
\begin{proof}
The span of the unit-interval indicators $\one_{(k-1,k)}$, $k\ge1$, in $L_{2,q}(0,\infty)$ is naturally equivalent to $\ell_{2,q}$. An embedding as displayed would therefore restrict to an embedding prohibited by Theorem~\ref{thm:endpoint}.
\end{proof}

\begin{remark}
The subcritical argument in Section~5, Theorem~\ref{thm:above2}, and Theorem~\ref{thm:endpoint} together prove the full statement of Theorem~\ref{thm:main}. In particular, they complete the parameter strip $0<q\le1$: for every $p>0$ with $p\neq q$,
\[
\ell_{p,q}\hooknot L_{p,q}(0,1),
\qquad
L_{p,q}(0,\infty)\hooknot L_{p,q}(0,1).
\]
For $q=1$ the endpoint conclusion recovers the known Banach case in \cite{SadovskayaSukochev}; the genuinely new endpoint content is $p=2$, $0<q<1$. The logarithmic weights $k^{-1/2}$ explain why the ordinary unweighted common-profile estimate is inconclusive at the quadratic threshold.
\end{remark}

Combining the results established above with the corresponding results available in the literature, we obtain the following corollary.

\begin{corollary}
Let $0<p, q<\infty$. Then
$
\ell_{ p, q} \hookrightarrow L_{p, q}(0,1)
$
 if and only if $p=q$.
\end{corollary}

\appendix
\section{Structural checks in the quasi-Banach range}
We collect here the standard Lorentz-lattice facts needed to justify the two structural ingredients used above: the mixed-norm estimate in Lemma~2.1 and the subsequence splitting principle in Section~4. The point is to verify that both remain available throughout the full range $0<p,q<\infty$, including the nonlocally convex cases.

\begin{proposition}
	Let $0<p,q<\infty$. If $0<\alpha<\min\{p,q,1\}$, then
	\[
	\bigl\|\operatorname{sgn}(f)|f|^{\alpha}\bigr\|_{p/\alpha,q/\alpha}
	=\|f\|_{p,q}^{\alpha}.
	\]
	Consequently, $L_{p,q}$ has nontrivial lattice convexity. Moreover, for every $r>\max\{p,q\}$,
	\[
	\left(\sum_j\|f_j\|_{p,q}^{\,r}\right)^{1/r}
	\le C_{p,q,r}
	\left\|\left(\sum_j|f_j|^r\right)^{1/r}\right\|_{p,q},
	\tag{A.1}
	\]
	so $L_{p,q}$ satisfies the corresponding lower $r$-estimate on disjoint families. Finally, the quasi-norm of $L_{p,q}(0,1)$ is order continuous and hence absolutely continuous; in particular, bounded simple functions are dense in $L_{p,q}(0,1)$.
\end{proposition}

\begin{proof}
	Since $(|f|^{\alpha})^*=(f^*)^{\alpha}$, the first identity follows directly from the definition of the Lorentz quasi-norm. Because $p/\alpha>1$ and $q/\alpha>1$, the usual convexification argument yields nontrivial lattice convexity. Estimate~(A.1) is the standard $r$-concavity inequality for Lorentz lattices when $r>\max\{p,q\}$, and its restriction to disjoint families gives the required lower $r$-estimate. These facts are classical; see Bennett and Sharpley~\cite[Chapter~2]{BennettSharpley}. For $q<\infty$, order continuity of the Lorentz quasi-norm is standard as well; on $(0,1)$ it yields absolute continuity, density of bounded simple functions, and separability~\cite[Chapter~2]{BennettSharpley}.
\end{proof}

\begin{remark}
	The role of Proposition~A.1 is to verify the hypotheses behind the quasi-Banach splitting argument, rather than to introduce new structural properties of Lorentz spaces. After an equivalent lattice renorming, the relevant convexity and lower-estimate constants may be normalized as required in Randrianantoanina's subsequence splitting theorem~\cite[Theorems~3.1 and~3.9]{Randrianantoanina2002}. Hence the splitting principle used in Section~4 applies for every $0<p,q<\infty$. In addition, (A.1) is precisely the norm-interchange estimate used in Lemma~2.1. Thus neither of these two ingredients carries a hidden local-convexity assumption such as $p\ge1$ or $q\ge1$.
\end{remark}

\section*{Declaration of competing interest}
The author declares no competing interest.

\section*{Data availability}
No data were used for the research described in this article.

\section*{
Acknowledgments}

Thanks to all the members of the Functional Analysis Research Team at the School of Mathematics and Statistics, Anqing Normal University, for their valuable discussions and corrections
regarding the challenges and errors encountered in this article.

\end{document}